\documentclass[11pt,a4paper]{amsart}
\usepackage[utf8]{inputenc}
\usepackage[pdftex]{graphics}

\usepackage[english]{babel}
\usepackage{tikz}
\usepackage{enumitem}
\usepackage{amsmath,amssymb, amsfonts, latexsym, stmaryrd}
\usepackage{amsthm,amscd}
\usepackage[foot]{amsaddr}

\newcommand{\sse}{\subseteq}
\newcommand{\N}{\mathbb{N}}
\newcommand{\ra}{\rightarrow}
\newcommand{\Clo}{\mathrm{Clo}}
\newcommand{\sM}{\mathcal{M}}
\newcommand{\sC}{\mathcal{C}}
\newcommand{\fP}{\mathfrak{P}}
\newcommand{\preserves}{\rhd}
\newcommand{\npreserves}{\ntriangleright}
\newcommand{\Ra}{\Rightarrow}
\newcommand{\set}[1]{\left\{ #1 \right\}}
\newcommand{\Maj}{\mathrm{Maj}}

\theoremstyle{plain}
\newtheorem{theorem}{Theorem}[section]
\newtheorem{corollary}[theorem]{Corollary}
\newtheorem{lemma}[theorem]{Lemma}

\date{\today}
\address{Institute of Algebra, Technische Universität Dresden, 01069 Dresden, Germany} 

\email{johannes-greiner@gmx.de}

\begin{document}

\title{Generating clones with conservative near-unanimity operation}

\author{Johannes Greiner}

\maketitle

\begin{abstract}
 Due to the Baker-Pixley theorem we know that every clone over a finite domain $A$ containing a near-unanimity operation $g$ is finitely generated.  Therefore there exists an integer $k$ such that the clone is generated by its $k$-ary part. In this paper we are interested in the size of $k$ for a fixed $A$ and fixed arity of a conservative $g$. We obtain lower bounds for all arities and they turn out to be sharp for arity three. 
\end{abstract}

\section{Introduction}
Let us define $A:=\set{0,1,2,\ldots, |A|-1} \sse \N_0$ to be a finite set.
For $d\in \N$, $d\geq 2$ an operation $g\colon A^{d+1} \ra A$ is said to be a \emph{near-unanimity operation} (nu-operation) if it reflects arguments which are unanimous except for one argument, which means for all $x,y\in A$ we have
\begin{displaymath}
g(x, \ldots, x,y)= g(x, \ldots, x,y,x)= \ldots = g(y,x, \ldots,x) = x.
\end{displaymath}
If the arguments $x_1, \ldots, x_{d+1}$ of a function follow this pattern they will be called \emph{near-unanimous} and the value in prevalence ($x$ here) will be denoted by $\Maj(x_1, \ldots, x_{d+1})$ in that case.\\

Now recall the well-known result by Baker and Pixley:
\begin{theorem}[\cite{BakerPixleyPolyInterpol}]
If a clone $C\leq O_A$ contains a $(d+1)$-ary near-unanimity operation, then there exists $k\in \N$ such that the $k$-ary operations of $C$ are sufficient to generate $C$.
\end{theorem}
Therefore we can consider $\lambda(C) := \min \{ k\in \N \mid C = \Clo (C^{(k)}) \}$ describing the minimal $k$ such that the $k$-ary part of $C$ characterizes the entire clone. It is a consequence of the theorem that the set $\sM^d_n$ of all clones on $A:=\set{0, \ldots, n-1}$ containing a $(d+1)$-ary nu-operation $g$ is finite. Therefore we can define
\begin{displaymath}
 \lambda_d(n) := \max\{ \lambda(C) \mid C \in \sM^d_n\}.
\end{displaymath}
Knowledge about $\lambda_d(n)$ can be used to examine and characterize unknown clones or to calculate them more efficiently. Furthermore there is the special case where $g$ is conservative, that is for all $x_1, \ldots, x_{d+1} \in A$ we have  $g(x_1, \ldots, x_{d+1}) \in \{x_1, \ldots, x_{d+1}\}$. The subset of $\sM^d_n$ containing a conservative $g$ is termed $\sC^d_n$ here and the corresponding function is $\gamma_d(n):= \max\{ \lambda(C) \mid C \in \sC^d_n\}$. \\
Work on $\lambda_d$ and $\gamma_d$ started in 1989 by Harry Lakser. Here is an overview of known results:

\begin{itemize}
 \item \cite{laskerFinGenCoT}:  $\lambda_2(n) = n(n-2)$ for $n\geq 5$
 \item \cite{kerkhoffZhukDraft}: $\lambda_2(4) = 8$, $\lambda_2(3) = 5$, $\lambda_2(2) = 3$ 
 \item \cite{kerkhoffOnTGenoCcNuOp}: $\lambda_d(n) = (n-1)^d-1$ for $n\geq (d-1)2^d +d+1$
 \item \cite{kerkhoffOnTGenoCcNuOp}:  $\lambda_d(n) \geq (n-1)^d-1$ for $n\geq 3$
 \item \cite{kerkhoffOnTGenoMCcNuOp}: $\gamma_d(n) \leq dn^{d-1}$
\end{itemize}

Note that while $\lambda_d(n)$ is known to a great extent, $\gamma_d(n)$ remains to be discovered. The latter has been introduced by Kerkhoff in \cite{kerkhoffOnTGenoMCcNuOp}.\\
In Section \ref{sec_Boundgamma_2_n} we will prove a lower bound on $\gamma_2(n)$, giving sharp results. The paper concludes with a generalization of the ideas for $d\geq 3$.\\
Before we begin, we would like to thank Dr. Sebastian Kerkhoff for his great support regarding this paper and the anonymous reviewer for his/her efforts and the helpful report.

\section{Preliminaries}

Let be $O_A$ the set of all finitary operations on set $A$. For $F\sse O_A$ and $k\in \N$ let $F^{(k)}$ be the set of all $k$-ary operations in $F$.  $F\sse O_A$ is called \emph{clone on $A$} if it includes all projection maps and is closed under composition, that is
\begin{displaymath}
\forall k, k'\in \N\, \forall f\in F^{(k)}, g_1, \ldots, g_k \in F^{(k')}\colon f(g_1, \ldots, g_k) \in F 
\end{displaymath}
where $f(g_1, \ldots, g_k)(x_1, \ldots, x_{k'}):=f(g_1(x_1,\ldots, x_{k'}), \ldots, g_k(x_1, \ldots, x_{k'}))$.\\
Since $O_A$ is a clone and intersections of clones are clones, we can define $\Clo(F)$ to be the smallest clone that contains $F$. We say that $F$ \emph{generates} $\Clo(F)$. Since clones on sets with one element are trivial, we will from now on consider $|A| \geq 2$ only. \\

A $l$-ary relation is an element of $\fP(A^l)$. We describe relations as matrices by interpreting their columns as elements of the relation. The union of matrices will therefore describe the set of all columns in the matrices, e.g. 
\begin{displaymath}
  \left( \begin{array}{ccc}
          1 & 2 & 2 \\
	  1 & 1 & 2
         \end{array} \right) \cup
         \left( \begin{array}{cc}
          2 & 2 \\
	  0 & 1
         \end{array} \right) = \left\{ \left( \begin{array}{c} 1 \\ 1 \end{array} \right), \left( \begin{array}{c} 2 \\ 0 \end{array} \right), \left( \begin{array}{c} 2 \\ 1 \end{array} \right), \left( \begin{array}{c} 2 \\ 2 \end{array} \right) \right\} .
\end{displaymath}

Note the ambiguity of this notation when it comes to a single column vector. On the one hand it can represent a relation with one element and on the other hand it can describe an element of a relation. We promise that the context will always clarify what is meant.\\
Whenever we have a fixed relation $\sigma$ and need to consider some matrix representation, we will choose the matrix $\Sigma$ containing the elements of $\sigma$ in lexical order\footnote{Any linear order is fine - we just need to fix one.}. With that we know what the \emph{$i$-th row of $\sigma$} means, namely $\Sigma_{(i,\cdot)}$\;.\\
For a relation $\rho$ of arity $l$ and an $d$-ary operation $f$ define
\begin{displaymath}
f(\rho):= \left\{\left. \left(\hspace*{-3pt} \begin{array}{c}
        f(r_1(1),\ldots , r_d(1)) \\
        \vdots \\
        f(r_1(l),\ldots , r_d(l))
          \end{array}\hspace*{-3pt} \right)  \;\right| \;
          \left(\hspace*{-3pt} \begin{array}{c}  r_1(1)\\         \vdots \\         r_1(l) \end{array} \hspace*{-3pt}\right), \hdots,  
          \left(\hspace*{-3pt} \begin{array}{c}  r_d(1)\\         \vdots \\         r_d(l) \end{array}\hspace*{-3pt} \right)
          \in \rho \right\}.
\end{displaymath}
Thinking of matrices, $xI^{(l)}$ (for $x\in A$) will denote the $l\times l$ matrix with $x$ on the diagonal and all other entries zero, which describes the relation
\begin{displaymath}
\left\{ \left( \begin{array}{c}x\\ 0\\0\\\vdots \end{array}\right), \left( \begin{array}{c}0\\x\\ 0\\ \vdots \end{array}\right), \left( \begin{array}{c}0\\0\\x\\ \vdots \end{array}\right), \ldots, \left( \begin{array}{c}0\\ \vdots \\ 0\\x \end{array}\right) \right\}.
\end{displaymath}

The $i$-th variable of an operation $f\in O_A^{(d)}$ is called \emph{essential} if there exist $x_1,\ldots, x_d, x_i'\in A$ such that $f(x_1, \ldots, x_i, \ldots, x_d) \neq f(x_1, \ldots, x_i', \ldots, x_d)$. We can now count the number of essential variables for an operation $f$ and denote the number by $\mathrm{ess}(f)$.\\
$f\in O_A$ is said to \emph{preserve} a relation $\rho$ if $f(\rho) \sse \rho$. The statement ``$f$ preserves $\rho$'' will be written as $f\preserves \rho$ and its converse as $f\npreserves \rho$.

\section{Lower bound for arity three}\label{sec_Boundgamma_2_n}
To prove a lower bound for $\gamma_2(n)$, as well as for Theorem \ref{thm:BoundGamma_d_n}, we will need the following Lemma.

\begin{lemma}\label{lm:essTool}
  Let $k$, $d$ and $n$ be positive integers, and let $A$ be an $n$-element set. Assume there exist relations $\sigma$ and $\rho$ on $A$ and $f,g\in O_A$ such that the following conditions hold:
  \begin{enumerate}[label=(\arabic*)]
  \item $|\sigma| = k$, $\sigma \sse \rho$,
  \item $f \npreserves \rho$ but $f\preserves \rho \setminus \set{t}$ for all $t\in \sigma$,
  \item $g$ is a $(d+1)$-ary conservative near-unanimity operation and $g \preserves \rho \setminus \set{t}$ for all $t\in \sigma$.
  \end{enumerate}
Then $\gamma_d(n) \geq k$.
\end{lemma}

\begin{proof}
  Assume that $h\in \Clo(\set{f,g})^{(k-1)}$ and let $r_1, \ldots, r_{k-1} \in \rho$. Since $|\sigma|=k$ there exists $t\in \sigma$ which is different from all $r_1, \ldots, r_{k-1}$. By condition (2) and (3) $f$ and $g$ preserve $\rho\setminus \set{t}$. But then $h$ must preserve $\rho \setminus \set{t}$ since every member of a clone preserves all relations its generators preserve. Therefore $h(r_1, \ldots, r_{k-1}) \in \rho \setminus \set{t} \sse \rho$ for arbitrary $r_1, \ldots, r_{k-1} \in \rho$, which means $h\preserves \rho$.\\
Hence $\Clo(\set{f,g})$ cannot be generated by its $k-1$-ary part, since all elements of $\Clo(\set{f,g})^{(k-1)}$ preserve $\rho$ but $f$ does not. Since $\Clo(\set{f,g})$ is a clone containing a $(d+1)$-ary conservative nu-operation, we have $\gamma_d(n) \geq k$.
\end{proof}

Using this lemma we can prove the first main result:
\begin{theorem} \label{lem_Boundgamma_2_n}
 $\gamma_2(n) \geq 2n $ for $n\geq 4$.
\end{theorem}

\begin{proof}
For given $n\geq 4$, we set $d=2$ and $k=2n$ and construct relations $\sigma, \rho$ and operations $f, g$ such that conditions (1) to (3) of Lemma \ref{lm:essTool} are satisfied.

Consider the following:
\begin{align*}
\sigma &:= \bigcup \left\{ 1I^{(2)}, \ldots, (n-1)I^{(2)} \right\} \cup \left( \begin{array}{cc} 
2 & 1 \\
1 & 2
\end{array} \right), \\  
\rho   &:= \sigma \cup \left( \begin{array}{cc}
0 \\
0  \end{array} \right), \\ 
f(x_1, \ldots, x_{2n}) &:= \left\{ \begin{array}{ll}
n-1 & \text{if } (x_1, \ldots, x_{2n}) \text{ equals the first row of } \sigma ,\\
n-1 & \text{if } (x_1, \ldots, x_{2n}) \text{ equals the second row of }\sigma , \\
0   & \text{otherwise},  \end{array} \right.  \\ 
g(x_1,x_2, x_3) &:= \left\{ \begin{array}{ll}
                     \Maj(x_1, x_2, x_3) & \text{if } (x_1,x_2, x_3) \text{ is near-unanimous},  \\
		     g^*(x_1,x_2, x_3) & \text{otherwise},
                     \end{array} \right.\\
g^*(x_1,x_2, x_3) &:= \left\{ \begin{array}{ll}
                     0 & \text{if } 0 \in \{x_1,x_2, x_3\}, \\
                     \max\{x_1,x_2, x_3\} & \text{otherwise.}
                     \end{array} \right.                      
\end{align*}
The following visualization of $\sigma$ and $\rho$ (to be read like an incidence matrix) should facilitate the understanding:
\begin{center}
 \begin{tikzpicture}[scale=0.9]
\draw[step=1,color=gray] (0,0) grid (7,7);

\foreach \i in {2,...,5,7}
{
  \draw (\i-0.5,6.5) node {$\sigma, \rho$};
}
\foreach \i in {5,...,2,0}
{
  \draw (0.5,\i+0.5) node {$\sigma, \rho$};
}

  \draw (6-0.5,6.5) node {$\ldots$};
  \draw (0.5,2-0.5) node {$\vdots$};
  
  \draw (0.5,6.5) node {$\rho$};
  
  \draw (1.5,4.5) node {$\sigma, \rho$};
  \draw (2.5,5.5) node {$\sigma, \rho$};
  
\foreach \i in {0,...,4}
{
  \draw (\i+0.5,-0.5) node {\i};
  \draw (-0.5,6.5-\i) node {\i};
}
  \draw (5+0.5,-0.5) node {$\ldots$};
  \draw (6+0.5,-0.5) node {$n-1$};
  
  \draw (-0.5,6.5-5) node {$\vdots$};
  \draw (-0.5,6.5-6) node {$n-1$};

\end{tikzpicture}
\end{center}
Now we start checking conditions (1) to (3), where (1) is trivial.\\

(2): Since $f(\rho) \ni  \bigl(\begin{smallmatrix}  n-1\\ n-1 \end{smallmatrix} \bigr)  \not \in \rho$ we have $f \npreserves \rho$. \\
Assume we have $r_1, \ldots, r_{2n} \in \rho \setminus \{ t \}$ for some $t \in \sigma$. Then we can never get $\binom{n-1}{n-1}=f(r_1, \ldots, r_{2n})$ because we would either need all entries from the diagonal (that is $\binom{0}{0}, \binom{1}{1},\binom{2}{2}$) in order to use the first (or second) case in $f$ only, or all elements of $\sigma$ to use the first and the second case in $f$. Neither is possible.\\
Therefore $f(\rho\setminus \{t\}) \sse \bigl(\begin{smallmatrix}  0 & n-1 & 0\\ 0 & 0 & n-1 \end{smallmatrix} \bigr)$ for any $t\in \sigma$. \\
If $t = \binom{n-1}{0}$ we have $\binom{n-1}{0} \not \in f(\rho\setminus \{t\})$ since the inclusion would require some element with $n-1$ in the first row, but there is only $t$ with that property. The analogue is true for $t = \binom{0}{n-1}$. Therefore $t\not\in f(\rho\setminus\set{t})$ for any $t\in\sigma$, which implies (2).\\

(3): Conservativity, near-unanimity and arity of $g$ are straightforward to check. It remains the show that  for all $t \in \sigma$ we have $g \preserves \rho \setminus \{t\}$ and we begin by proving $g(\rho) \sse \rho$.\\
Observe that 1 (and 2) can only be a result from the nu-part of $g$ because three non-near-unanimous, non-zero arguments must have a maximum greater than 2.\\
Assume we have $g(r_1, r_2, r_3)$ equaling $\binom{1}{1}$ (or $\binom{2}{2}$). Then we would need two arguments to be 1 (or two being 2) in the first and second line, which implies to use $\binom{1}{1}$ (or $\binom{2}{2}$) to get the result $\binom{1}{1}$ (or $\binom{2}{2}$). Hence $\binom{1}{1}$, $\binom{2}{2} \not \in g(\rho)$.\\ 
Let be $a,b\in \N, a\geq 3, b\neq 0$ and assume we have $r_1, r_2, r_3 \in \rho$ such that $\binom{a}{b} = g(r_1, r_2, r_3)$. Then we must have some $r_i$ such that $r_i = \binom{a}{0}$ because it is the only element with $a$ in the first coordinate. Without loss of generality we will choose $i=1$. But then $a$ needs to be a result of $g^*$ since $b=0$ otherwise. Hence, $a$ is the greatest integer in the first row and all arguments there are pairwise distinct and nonzero. For example
\begin{center}
$\begin{array}{rclcccl}
 a &=& g^*(&a, &a-1, & a-2 &), \\
 b &=& g(  &0, &?,  & ?    &).
\end{array}$ 
\end{center}
Now observe that there are at most two different elements in $\rho$ with $b$ in the second coordinate, one of which always has 0 in the first coordinate. Therefore $b$ must be a result of $g^*$ as well, because otherwise we would need one entry to be 0 in the first row. But then $b=0$ because there is a zero below $a$. Contradiction. Therefore $a\geq 3$ implies $b=0$.\\
$ $\\
Because $\sigma$ and $\rho$ are symmetric relations, the analogue is true if we swap the restrictions for $a$ and $b$. Hence if $\binom{a}{b} \in g(\rho)$ we have $\binom{a}{b} \in \rho$, i.e. $g(\rho) \sse \rho$.\\ 

It remains to show that for all $t \in \sigma$ we have $t \not \in g(\rho \setminus \set{t})$. For reasons of symmetry of $\rho$, the following cases are sufficient: 

\begin{itemize}
\item $t=\binom{0}{1}$: As argued above, 1 can only be produced through the nu-part of $g$. Notice that the only elements of $\rho$ with 1 in the second coordinate are $\binom{0}{1}$ and $\binom{2}{1}$. Therefore if $\binom{?}{1} = g(r_1, r_2, r_3)$ with $r_1, r_2, r_3 \in \rho \setminus \set{t}$, there must be two arguments having  2 in the first coordinate. But then the result of $g$ will be $\binom{2}{1} \neq t$, yielding  $\binom{0}{1} \not \in g(\rho \setminus \set{\binom{0}{1}})$.
\item $t=\binom{0}{2}$: Swapping 1 and 2 in the above case gives $\binom{0}{2} \not \in g(\rho \setminus \set{\binom{0}{2}})$.
\item $t=\binom{1}{2}$: Again, 1 and 2 can only be produced by the nu-part of $g$. But if $r_1, r_2, r_3 \in \rho \setminus \set{t}$ and $\binom{1}{2} = g(r_1, r_2, r_3)$ then there need to be two arguments with 1 in the first coordinate and two with 2 in the second coordinate implying that there exists $r_i = \binom{1}{2}$ for some $i\in \set{1,2,3}$. Contradiction. Therefore $\binom{1}{2} \not \in g(\rho \setminus \set{\binom{1}{2}})$.
\end{itemize}
Hence $g\preserves \rho\setminus \set{t}$ for all $t\in \sigma$.\\
Since our construction works for all $n\geq 4$ and Lemma \ref{lm:essTool} is applicable, we conclude $\gamma_2(n) \geq 2n$ for $n\geq 4$.
\end{proof}

The framework used above is a condensed version of some of the ideas used in \cite{kerkhoffZhukDraft}. Combined with the results from \cite{kerkhoffOnTGenoMCcNuOp}, Theorem \ref{lem_Boundgamma_2_n} immediately yields the following corollary.

\begin{corollary}
  $\gamma_2(n) = 2n$ for all $n\geq 4$.
\end{corollary}

Lastly we have $\gamma_2(2)=\lambda_2(2)$ and $\gamma_2(3) = \lambda_2(3)$ because every nu-operation on a set with three or less elements is necessarily conservative.

\section{Generalization for higher arities}
Unfortunately, we were not able to generalize our sharp lower bound to higher arities $d$ in such a way that they remain at least 'potentially sharp'. \\

However, the following theorem is a first result for $d\geq 3$ using the same techniques as above:
\begin{theorem} \label{thm:BoundGamma_d_n}
 $\gamma_d(n) \geq d(n-2)$ for $n\geq 3, d\geq 3$.
\end{theorem}

\begin{proof}
 The proof is very similar to the proof of Theorem \ref{lem_Boundgamma_2_n} and we will use Lemma \ref{lm:essTool} again.\\
 For given $n$ and $d$ define:
 \begin{align*}
 k &:=d(n-2),\\
 \sigma &:= \bigcup \left\{ 1I^{(d)}, \ldots, (n-2)I^{(d)} \right\},  \\ 
\rho   &:= \sigma \cup \{0,n-1\}^d \setminus \{n-1\}^d,  \\ 
f(x_1, \ldots, x_{k}) &:= \left\{ \begin{array}{ll}
n-1 & \text{if } \exists i \in \set{1, \ldots, d}\colon \\
 & 			\quad (x_1, \ldots, x_{k}) \text{ equals the } i\text{-th row of } \sigma, \\
0   & \text{otherwise},  \end{array} \right.  \\ 
g(x_1, \ldots, x_{d+1}) &:= \left\{ \begin{array}{ll}
                     \Maj(x_1, \ldots, x_{d+1}) & \text{if } (x_1, \ldots, x_{d+1}) \\
						  & \quad \text{is near-unanimous},  \\
		     g^*(x_1,\ldots, x_{d+1}) & \text{otherwise},
                     \end{array} \right.\\
g^*(x_1,\ldots, x_{d+1}) &:= \left\{ \begin{array}{ll}
                     0 & \text{if } 0 \in \{x_1,\ldots, x_{d+1}\}, \\
                     \max\{x_1,\ldots, x_{d+1}\} & \text{otherwise.}
                     \end{array} \right.                      
\end{align*}
We will now check, that these relation and operation satisfy condition (1) to (3) from Lemma \ref{lm:essTool}, where (1) is trivial again.\\

(2): First of all we have $f(\rho) \supseteq \{n-1\}^d \not\sse \rho$ which means $f \npreserves \rho$. Now choose $t \in \sigma$ and observe that there is $a\in \{1, \ldots, n-2\}$ and $i \in \{1, \ldots, d\}$ such that $t$ is the unique element of $\sigma$ having $a$ in its $i$-th coordinate. Then $f(\rho \setminus \{t \})$ cannot contain any element with $n-1$ in the $i$-th coordinate because we need all numbers $1, \ldots, n-2$ to appear in the argument if we want to produce $n-1$ with $f$. Therefore $\{n-1\}^d \not \sse  f(\rho \setminus \{t \})$. Hence for all $t \in \sigma$ it holds that $f\preserves \rho \setminus \{t \}$ because any tuple consisting of 0 and $n-1$ except for the one in $\set{n-1}^d$ is in $\rho$. \\

(3): Again conservativity, near-unanimity and arity of $g$ are easy to check. The rest can be split into the following steps:
\begin{enumerate}[label*=3.\arabic*]
 \item $g(\rho) \sse \rho$
 \begin{enumerate}[label*=.\arabic*]
  \item $\set{n-1}^d \not \sse g(\rho)$
  \item If $r_1, \ldots, r_{d+1} \in \rho$ and $s = g(r_1, \ldots, r_{d+1})$ and $\exists i\colon s(i) = n-1$ then $\forall j\in \set{1, \ldots, d}\colon s(j)\in \set{0,n-1}$
  \item If $r_1, \ldots, r_{d+1} \in \rho$ and $s = g(r_1, \ldots, r_{d+1})$ and $\exists i\colon s(i) \in \set{1, \ldots,n-2}$ then $\forall j \neq i\colon s(j) = 0$
 \end{enumerate}
 \item $t \in \sigma \Ra t \not \in g(\rho \setminus \set{t})$
\end{enumerate}

3.1.1: Pick any $s$ in $g(\rho)$. Then there exist $r_1, \ldots, r_{d+1} \in \rho$ such that $s=g(r_1, \ldots, r_{d+1})$. Since there is at least one zero in each $r_i \in\rho$, we have at least $d+1$ zeros in the $d\times (d+1)$ matrix given by $(r_1, \ldots, r_{d+1})$ and therefore at least one row $j$ with two zero entries (WLOG we choose the first two). But then $s(j) = g(0,0, \ldots) = 0$. Therefore $\{n-1\}^d \not \sse g(\rho)$.\\

3.1.2: Pick an arbitrary index $j\neq i$. If the value of $s(i)$ is produced through the nu-case, there exist distinct indices $\iota_1, \iota_2$ such that $r_{\iota_1}(i)=r_{\iota_2}(i)=n-1$ and therefore $r_{\iota_1}(j), r_{\iota_2}(j) \in \set{0, n-1}$ which implies $s(j)\in \set{0, n-1}$.\\
If, however, the value of $s(i)$ is produced through $g^*$, we would need at least two nonzero, non-$(n-1)$ entries in row $i$ (here $n-2$ and $n-3$) because any zero would imply $s(i)=0$ and all but one entry being $n-1$ would imply that $s(i)$ is produced through the nu-case. This results in two zeros in row $j$.
\begin{center}
$\begin{array}{rlcl}
&&\vdots &\\
 s(i) &=n-1 	&=& g^*(n-1, n-2, n-3, \ldots) \\
&& \vdots &\\
s(j) &= 0 	&=& g(\{0, n-1\},\; 0,\quad 0,  \ldots)\\
&& \vdots &
\end{array}$
\end{center}
But then $s(j)=0$.\\

3.1.3: If $s(i)\in \set{1, \ldots, n-2}$ either $s(i)=\max\set{r_1(i), \ldots, r_{d+1}(i)}$ with $0 \not\in \set{r_1(i), \ldots, r_{d+1}(i)}$ or all but one of the arguments equal $s(i)$. In both cases we get\footnote{WLOG we chose the last column to be the unspecified one.}: 
\begin{center}
$\begin{array}{rclccccl}
s(1) &= &g(  &0 &\ldots &0 &? &) \\ 
\vdots & &  &\vdots & &\vdots & &) \\
s(i-1) &= &g(  &0 &\ldots &0 &? &) \\
s(i) &= &g(  &r_1(i) &\ldots &r_d(i) &r_{d+1}(i) &) \\
s(i+1) &= &g(  &0 &\ldots &0 &? &) \\
\vdots & &  &\vdots & &\vdots & &) \\
s(d) &= &g(  &0 &\ldots &0 &? &).
\end{array}$
\end{center}
Therefore all other $s(j)$ for $j\in \set{1,\ldots,d}\setminus \set{i}$ must be zero.\\

3.2: Choose $t\in \sigma$. Then there are $i\in \set{1, \ldots, d}$ and $a\in\set{1, \ldots, n-2}$ with $t(i) = a$ and only $t$ satisfies this equation. By conservativity of $g$ and uniqueness of $t$ it follows that $t\not \in g(\rho\setminus \set{t}) $ because $g$ would need $a$ in the $i$-th line of arguments to satisfy $g(\ldots)(i)=a$. \\

Therefore (3) is satisfied and Lemma \ref{lm:essTool} yields $\gamma_d(n) \geq d(n-2)$ for $n\geq 3$ and $d\geq 3$.
\end{proof}

Even though this generalizes some of the techniques used in the previous section, we think that this is by no means sharp. It seems that the general case needs relations and operation with a much 'deeper' interplay or an entirely different approach. Finding such $f,g,\rho$ and $\sigma$ and understanding their interplay is an open problem.

\section{Summary}
In this paper we have derived that any clone on a finite set with $n$ elements containing a conservative 3-ary near-unanimity operation can be generated by its $2n$-ary part if $n\geq 4$. For $n=3$ the 5-ary and for $n=2$ the 3-ary part is sufficient.\\
Furthermore we obtained lower bounds for conservative near-unanimity operations of higher arity.

\bibliographystyle{amsplain} 
\bibliography{lit_generatingClones}

\end{document}